\documentclass[12pt]{amsart}
\usepackage{hyperref}
\hypersetup{nesting=true,debug=true,naturalnames=true}

\usepackage[english]{babel}

\usepackage{amsmath,amssymb,amsfonts}
\usepackage{amscd, amssymb, latexsym, amsmath, amscd}

\usepackage{anysize}
\marginsize{3cm}{3cm}{3cm}{3cm}

\usepackage{cleveref}
\usepackage{tikz}
\usetikzlibrary{decorations.pathreplacing, patterns}
\usepackage{subcaption}

\newtheorem{theorem}{Theorem}[section]
\newtheorem{lemma}[theorem]{Lemma}

\newtheorem{corollary}[theorem]{Corollary}
\newtheorem{proposition}[theorem]{Proposition}
\newtheorem{remark}[theorem]{Remark}

\newcommand{\Z}{\mathbb{Z}}

\newcommand{\C}{\mathbb{C}}

\def\SL{{\rm SL}}
\def\GL{{\rm GL}}
\def\Sym{{\rm Sym}}

\author{Dibyendu Biswas}
\address{
	Indian Institute of Technology, Kanpur, India, 208016}
\email{dibubis@gmail.com, dibyendub@iitk.ac.in}



\begin{document}

\title[Image of regular unipotents]
{Image of regular unipotent under a representation of $\GL_3(\C)$}

\begin{abstract}
We study the image of a regular unipotent element under any finite-dimensional irreducible polynomial representations of $\GL_3(\C)$. This problem is equivalent to decomposing certain compositions of irreducible representations as $\mathrm{SL}_2(\mathbb{C})$-modules. We give an explicit decomposition of this finding, its Jordan decomposition.
\end{abstract}

\maketitle  
 {\hfill \today} 

\section{Introduction}
	Let \(G\) be a connected reductive algebraic group over \(\mathbb{C}\). An element \(u \in G\) is called \emph{unipotent} if \(\rho(u)\) is unipotent for some (equivalently, any faithful) finite-dimensional rational representation \(\rho:G\to\GL(V)\).

 Let
   $\pi_{\lambda}:G\to \GL(V)$ be a finite-dimensional polynomial representation of $G$ with highest weight $\lambda$. For any unipotent element $u\in G$, the image $\pi_{\lambda}(u)$ is unipotent.
     It is natural to ask: what is the unipotent class of $\pi_{\lambda}(u)$ in $\GL(V)$, or
     equivalently, what is the Jordan canonical form of $\pi_{\lambda}(u)$?

    In this paper, we 
   will only consider the case of  $u$,  a \textit{regular unipotent} element. See \cite{steinberg-ihes} for more details on regular unipotent elements.

For studying this problem, we use  principal $\mathrm{SL}_2$ inside $G$.  There exists a homomorphism $\psi: \mathrm{SL}_2(\mathbb{C}) \rightarrow \mathrm{G}$
which takes a regular unipotent in $\mathrm{SL}_2(\mathbb{C})$ to a regular unipotent element in G . Such a homomorphism is unique up to conjugacy, and is called the principal $\mathrm{SL}_2$ in G . Therefore determining the Jordan form of $\pi_\lambda(u)$ amounts to understanding the restriction of $\pi_\lambda$ to this principal $\mathrm{SL}_2$. Note that for $G=\GL_n(\C)$, principal homomorphism
is precisely $\operatorname{Sym}^{n-1}:\mathrm{SL}_2 (\C)\longrightarrow \mathrm{GL}_n( \mathbb C)$ which takes a regular unipotent in $\mathrm{SL}_2(\mathbb{C})$ to a regular unipotent element in $\mathrm{GL}_n( \mathbb C)$.

 We have explicit character formula of this restricted representation due to J.-P. Serre \cite[Theorem~2.8]{serre2025zeroscaracteres} as follows for which we follow the notation of \cite{dptorsion-unipotent}.
\begin{theorem}\cite{dptorsion-unipotent}\label{thm dp reductive char}
	Let G be a connected reductive group, with a maximal torus $T$ and let $\lambda \in X^*(\mathrm{~T})$ be a dominant weight and let  $\pi_\lambda$ be the irreducible representation of G with highest weight $\lambda$. Let $\Theta_\lambda$ be the character of $\pi_\lambda$, and $\Theta_\lambda(z)$ is the character at the diagonal element $\left(z, z^{-1}\right)$ of $\mathrm{SL}_2(\mathbb{C})$ treated as an element of G via principal $\mathrm{SL}_2$ embedding. Then we have
	$$
	 	\Theta_\lambda(z):=\mathrm{ch} \left(\pi_\lambda|_{\widetilde{\mathrm{SL}_2}}\right)=z^{-2\left\langle\lambda, \rho^{\vee}\right\rangle} \frac{\prod_{\alpha \in \Phi^{+}}\left(1-z^{2\left\langle\lambda+\rho, \alpha^{\vee}\right\rangle}\right)}{\prod_{\alpha \in \Phi^{+}}\left(1-z^{2\left\langle\rho, \alpha^{\vee}\right\rangle}\right)}.
	$$
\end{theorem}

Therefore, we decompose this character as an direct sum of $\SL_2(\C)$ characters, which will answer of our problem of finding Jordan normal form of $\pi_{\lambda}(u)$. Recall that the finite-dimensional irreducible representations of $\mathrm{SL}_2(\C)$ are precisely the symmetric powers $\mathrm{Sym}^j:=\mathrm{Sym}^j(\C^2)$, of dimension $j+1$, of the standard representation of $\SL_2$ for $j\ge 0$. The character of $\mathrm{Sym}^j$ of $\mathrm{SL}_2(\mathbb{C})$, evaluated at the diagonal element $\mathrm{diag}(z, z^{-1})$, is given by $\mathrm{ch} \bigl(\mathrm{Sym}^j\bigr)=z^j + z^{j-2} + z^{j-4} + \cdots + z^{-(j-2)} + z^{-j}$.
Hence if 
\begin{equation*}
 	\Theta_\lambda(z)= \sum_{i} m_i \; \mathrm{ch} \bigl(\mathrm{Sym}^{n_i}\bigr),
\end{equation*}
as $\SL_2(\C)$ characters, then the associated partition is $\mathfrak p=((n_1+1)^{m_1} , (n_2+1)^{m_2} , \ldots)$ of $\dim \pi_\lambda$ gives precisely the Jordan type of $\pi_\lambda(u)$.

In this paper we explicitly compute the decomposition for $G=\GL_3(\C)$ in \Cref{thm for gl3}.

\Cref{thm for gl3} is also of independent interest from the viewpoint of algebraic combinatorics, as it gives the decomposition of the plethysm of symmetric functions $s_{\lambda}[h_2](x,y)$, where $\ell(\lambda)\leq 3$.

\section{Preliminaries}
  Composing a representation $\pi_\lambda$ with the principal $\mathrm{SL}_2(\C)$ in $\mathrm{GL}_3(\C)$ gives us:
  \[
  \begin{array}{ccccc}
  	\mathrm{SL}_2(\C) & \xrightarrow{\ \operatorname{Sym}^{2}\ } & \mathrm{GL}_3(\mathbb C) & \xrightarrow{\ \pi_\lambda\ } & \mathrm{GL}\bigl(\mathbb S_\lambda(\mathbb C^3)\bigr) \\[0.4ex]
  	\begin{pmatrix}1&1\\0&1\end{pmatrix} & \longmapsto & u & \longmapsto & \pi_\lambda(u).
  \end{array}
  \]
  Thus $\pi_\lambda(u)$ is the image of a regular unipotent element of $\mathrm{SL}_2(\C)$ under this composition, and the problem reduces to decomposing $\pi_\lambda\bigl(\operatorname{Sym}^{2}\mathbb C^2\bigr)$ as an $\mathrm{SL}_2(\C)$-module for which
we use the character formula for the representation
  $\pi_\lambda\bigl(\operatorname{Sym}^{2}\mathbb C^2\bigr),$
  established in \cite[Theorem~3.1]{dptorsion-unipotent}. 
  
\subsection{Lemma} 
 We use the following equivalent formulation of $\Theta_\lambda(z)$ given in \Cref{thm dp reductive char}.
\begin{proposition}\label{prop dp char} The character given in \Cref{thm dp reductive char} can be written as 
	\begin{equation*}
		\Theta_\lambda(z)= \prod_{\alpha \in \Phi^{+}} \frac{\left(z^{\left\langle\lambda+\rho, \alpha^{\vee}\right\rangle}-z^{\left\langle \lambda+\rho, \alpha^{\vee}\right\rangle}\right)}{\left(z^{\left\langle\rho, \alpha^{\vee}\right\rangle}-z^{\left\langle \rho, \alpha^{\vee}\right\rangle}\right)};
	\end{equation*}
\end{proposition}
\begin{proof} We have the character formula
	$$
	\Theta_\lambda(z)=z^{-2\left\langle\lambda, \rho^{\vee}\right\rangle} \frac{\prod_{\alpha \in \Phi^{+}}\left(1-z^{2\left\langle\lambda+\rho, \alpha^{\vee}\right\rangle}\right)}{\prod_{\alpha \in \Phi^{+}}\left(1-z^{2\left\langle\rho, \alpha^{\vee}\right\rangle}\right)}= \prod_{\alpha \in \Phi^{+}} \frac{\left(z^{\left\langle\lambda+\rho, \alpha^{\vee}\right\rangle}-z^{-\left\langle \lambda+\rho, \alpha^{\vee}\right\rangle}\right)}{\left(z^{\left\langle\rho, \alpha^{\vee}\right\rangle}-z^{-\left\langle \rho, \alpha^{\vee}\right\rangle}\right)}.
	$$	
	Factoring out \(z^{\langle \lambda+\rho,\alpha^\vee\rangle}\) from each factor in the numerator and \(z^{\langle \rho,\alpha^\vee\rangle}\) from each factor in the denominator, as \(\alpha\) ranges over \(\Phi^+\), the total power of \(z\) becomes
	 $$-2\left\langle\lambda, \rho^{\vee}\right\rangle + \sum_{\alpha \in \Phi^{+}}   \left\langle\lambda+\rho, \alpha^{\vee}\right\rangle - \sum_{\alpha \in \Phi^{+}}   \left\langle \rho, \alpha^{\vee}\right\rangle =0.$$  The last equality is due to  
	$\rho^{\vee}=\frac{1}{2} \sum_{\alpha \in \Phi^{+}}   \alpha^{\vee}$.
\end{proof}

Recall that the character of $\mathrm{Sym}^j$ of $\mathrm{SL}_2(\mathbb{C})$, evaluated at the diagonal element $\mathrm{diag}(z, z^{-1})$, is given by $\mathrm{ch} \bigl(\mathrm{Sym}^j\bigr) = z^j + z^{j-2} + z^{j-4} + \cdots + z^{-(j-2)} + z^{-j}$.
For brevity, we henceforth identify each representation with its character and write
\begin{equation}\label{eqn sym^j}
	\mathrm{Sym}^j = z^j + z^{j-2} + \cdots + z^{-j}.
\end{equation}
We take the convention that $\mathrm{Sym}^j = 0$ for $j < 0$. Using this notation, we have the following proposition, which we will need later.

\begin{proposition}Let $m$ be an even positive integer. Then
	\begin{equation}\label{eqn: alter char for div 2 }
			\frac{z^{m}-z^{-m}}{z^2-z^{-2}}
			=\sum_{k=0}^{\frac{m-2}{2}} (-1)^{k} \Sym^{m-2-2k}.
	\end{equation}
\end{proposition}
\begin{proof}
	Using the identity $\frac{z^{m}-z^{-m}}{z^2-z^{-2}}
	= z^{m-2}+z^{m-6}+z^{m-10}+\cdots+z^{-(m-6)}+z^{-(m-2)},$
	and the relation $z^{j}+z^{-j}=\mathrm{Sym}^{j}-\mathrm{Sym}^{j-2}, j\ge 1,$
	we pair the terms symmetrically about the middle term in the above expansion. Substituting the latter identity into each pair and summing the resulting expressions yields the desired result.
\end{proof}

Before going to state our theorem for $\GL_3( \C)$, we need two lemmas.
\begin{lemma}\label{lem: prod sym with alternate sym}
	Let $m,n \in \Z_{\geq 0}$ and $n$ even.
	Let
	$$P(m,n)=\Sym^m \otimes \left[ \Sym^n -\Sym^{n-2} + \cdots + (-1)^{\frac{n}{2}-1} \Sym^2 + (-1)^{\frac{n}{2}} \Sym^0 \right]  .$$
	Then 
	\begin{enumerate}
		\item If $m \geq n$,
		$$P(m,n)=\Sym^{m+n} + \Sym^{m+n-4}+ \cdots + \Sym^{m-n}$$
		\item If $m < n$ and $m$ even,
		\[
		P(m,n)=
		\begin{cases}
			\displaystyle
			\sum_{k=0}^{\frac{m+n}{4}} \Sym^{m+n-4k}
			-\sum_{k=0}^{\frac{n-m-4}{4}} \Sym^{n-m-2-4k},
			& m+n\equiv 0 \pmod{4},
			\\[1em]
			\displaystyle
			\sum_{k=0}^{\frac{m+n-2}{4}} \Sym^{m+n-4k}
			-\sum_{k=0}^{\frac{n-m-2}{4}} \Sym^{n-m-2-4k},
			& m+n\equiv 2 \pmod{4}.
		\end{cases}
		\]
	\end{enumerate}
\end{lemma}
\begin{proof} For Part~(1), write the alternating sum $\sum_{k=0}^{\frac{n}{2}} (-1)^k \Sym^{n-2k}$
	by grouping consecutive terms as
$(\Sym^n-\Sym^{n-2})+(\Sym^{n-4}-\Sym^{n-6})+\cdots$.
	Since \(m\geq n\), the Clebsch--Gordan formula gives
	\[
	\Sym^m\otimes (\Sym^n-\Sym^{n-2})
	=
	\Sym^{m+n}+\Sym^{m-n}.
	\]
	Applying the same computation to each of the remaining pairs and summing the resulting expressions, we obtain
	\[
	\left(\Sym^{m+n}+\Sym^{m-n}\right)
	+
	\left(\Sym^{m+n-4}+\Sym^{m-n-4}\right)
	+\cdots .
	\]
	This proves the part(1).
	
	We only prove the case \(m+n \equiv 0 \pmod{4}\); the case \(m+n \equiv 2 \pmod{4}\) follows similarly in part(2). By assumption, \(m\) and \(n\) are both even with \(m<n\). When $k=\frac{n-m}{2}$,
	we have \(n-2k=m\). Hence, we may decompose the tensor product into two parts as follows:
	\begin{align*}
		P(m,n)&=\Sym^m \otimes \left[ \sum_{k=0}^{\frac{n}{2}} (-1)^{k} \Sym^{n-2k} \right] \\
		&=\Sym^m \otimes \left[ \sum_{k=0}^{\frac{n-m}{2}-1}  (-1)^{k} \Sym^{n-2k} + \sum_{k=\frac{n-m}{2}}^{\frac{n}{2}}  (-1)^{k} \Sym^{n-2k}\right] \\
		& =\left[ \displaystyle
		\sum_{k=0}^{\frac{n-m}{4}-1} \Sym^{m+n-4k}
		-\sum_{k=0}^{\frac{n-m}{4}-1} \Sym^{n-m-2-4k} \right]+  \left[ \displaystyle
		\sum_{k=\frac{n-m}{4}}^{\frac{m+n}{4}} \Sym^{m+n-4k} \right]
	\end{align*}
The first bracketed term in the last equality is obtained by taking the tensor product of \(\Sym^m\) with the first alternating sum, while the second bracketed term is obtained by taking the tensor product of \(\Sym^m\) with the second alternating sum. Hence completes the proof.
	 
\end{proof}
\Cref{lem: prod sym with alternate sym} is important to prove the \Cref{thm for gl3}.

For a concise formulation of the lemma and to simplify the notation, we write $[k+1]$ for $\Sym^k$, the $(k+1)$-dimensional irreducible representation of $\SL_2(\C)$. Recall from \eqref{eqn sym^j} that the same notation $[k+1]$ will also be used to denote the character of $\Sym^k$, whenever the context is clear.  Recall $\lfloor m \rfloor$ denotes the floor function of $m$. 
\begin{lemma}\label{lem: sym{l} with sym{l+2} onwards}
	Let $u,v$ be nonnegative integer. Assume that $v$ is odd.
	Consider
	$$Q(u,v)=\Sym^{u}  \otimes \left[ \Sym^{u+2} + \Sym^{u+6} + \cdots +  \Sym^{u+2v}  \right]$$
	Define
	\begin{equation*}
		s =u+v, \qquad
		r=2s+2, \qquad
		p = \Bigl\lfloor \tfrac{s}{2} \Bigr\rfloor \qquad \text{and} \qquad t=\Bigl\lfloor \tfrac{\min\{u,v\}}{2} \Bigr\rfloor.
	\end{equation*}
	
	Then 
	\begin{enumerate}
		\item If $u$ is odd, or if $u$ is even satisfying $u > v - 1$, then  $Q(u,v)$ equals to
				\begin{equation*}
				\sum_{k=1}^{t} k\!\left(
				[4k-1] + [4k+1] + [r-4k+1] + [r-4k+3]
				\right)
				+\;
				(t+1)\!\sum_{j=2t+1}^{s-2t} [2j+1].
			\end{equation*}
		\item If $u$ is even satisfying $u \leq v - 1$, then  $Q(u,v)$ equals to 
		\begin{equation*}
			\sum_{k=1}^{t} k\!\left(
			[4k-1] + [4k+1] + [r-4k+1] + [r-4k+3]
			\right)
			+\;
			t\!\sum_{j=2t+1}^{s-2t} [2j]
			+\;
			\sum_{j=t}^{p-t} [4j+3].
		\end{equation*}
	\end{enumerate}
\end{lemma}
\begin{proof}For part (1), we show it only when $u$ is odd; the rest can be done similarly.
	Let $Q(k)=\Sym^{u}\otimes \Sym^{u+2k}$, where \(k\) is odd, and let
	$Q(u,v)=\sum_{\substack{1\le k\le v\\ k\ \mathrm{odd}}} Q(k)$.
	By the Clebsch--Gordan formula,
	\[
	Q(k)=\Sym^{2k}\oplus \Sym^{2k+2}\oplus \cdots \oplus \Sym^{2u+2k},
	\]
	and hence \(Q(k)\) contains exactly \(u+1\) irreducible summands for every odd \(k\).
	Therefore, the multiplicities in \(Q(u,v)\) are obtained by counting the overlaps among the strings of consecutive (having gap 2) irreducible summands arising from the various \(Q(k)\). The first summand contributes $\Sym^{2}\oplus \Sym^{4}\oplus \cdots \oplus \Sym^{2u+2}$,
	the second contributes $\Sym^{6}\oplus \Sym^{8}\oplus \cdots \oplus \Sym^{2u+6}$,
	and each subsequent summand is obtained by shifting the entire string four units to the right.
	Consequently, the multiplicity in \(Q(u,v)\) is equal to the number of such strings containing it. As one moves from left to right, the number of overlapping strings increases by one every two terms until all the strings overlap simultaneously. At this stage, the multiplicity attains its maximal value, equal to the number of odd integers between \(1\) and \(v\). Thereafter, the same process occurs in reverse, and the multiplicities decrease symmetrically. This yields the claimed decomposition.
	Part(2) similarly follows.
\end{proof}

\section{Theorem}
 We now come to the main theorem proved in this paper.
 
Let \(G=\GL_3(\C)\) and let \(\{\alpha,\beta\}\) be the set of simple roots, and let \(\{\omega_1,\omega_2\}\) denote the fundamental weights and let $\lambda=a\omega_1+b\omega_2$, where $a,b \in \Z_{\geq 0}$. To decompose $\pi_\lambda\bigl(\operatorname{Sym}^{2}\mathbb C^2\bigr)$ as an $\mathrm{SL}_2(\C)$-module, it is enough to assume $a\ge b\ge0$.
	Since the desired character formula (see \eqref{eqn char formula for gl3}) is symmetric with respect to $a,b$.

 Recall the notation $[k+1]$ for $\Sym^k$, the $(k+1)$-dimensional irreducible representation of $\SL_2(\C)$.
 \begin{theorem}\label{thm for gl3}
Let \(G = \GL_3(\C)\).
Let $\lambda = a\omega_1 + b\omega_2$, for $a\geq b \ge 0$. Define
 	\begin{equation*}
 	s =a+b, \qquad
 	r=2s+2, \qquad
 	p = \Bigl\lfloor \tfrac{s}{2} \Bigr\rfloor,  \qquad \text{and} \qquad t=\Bigl\lfloor \tfrac{b}{2} \Bigr\rfloor.
 	\end{equation*}
 	
 	If  $b =0$, then $\pi_{\lambda}(\Sym^2)= \sum_{k=0}^{p} [2a-4k+1].$
 	
 	Suppose $b\geq1$. 
 	\begin{enumerate}
 		\item If $b$ is odd, then $\pi_{\lambda}(\Sym^2)$ equals to
 		\begin{equation*}
 			\sum_{k=1}^{t} k\!\left(
 			[4k-1] + [4k+1] + [r-4k+1] + [r-4k+3]
 			\right)
 			+\;
 			(t+1)\!\sum_{j=2t+1}^{s-2t} [2j+1].
 		\end{equation*}
 		
 		\item 
 		If $(a,b) \equiv (\text{odd},\text{even})$
 		Then $\pi_{\lambda}(\Sym^2)$ equals to
 		\begin{equation*}
 			\sum_{k=1}^{t} k\!\left(
 			[4k-1] + [4k+1] + [r-4k+1] + [r-4k+3]
 			\right)
 			+\;
 			t\!\sum_{j=2t+1}^{s-2t} [2j+1]
 			+\;
 			\sum_{j=t}^{p-t} [4j+3].
 		\end{equation*}

 		\item 
 		Suppose $(a,b) \equiv (\text{even},\text{even})$.
 		Then $\pi_{\lambda}(\Sym^2)$ equals to
 		\begin{equation*}
 			\sum_{k=1}^{t-1} k\!\left(
 			[4k+1] + [4k+3] + [r-4k+1] + [r-4k+3]
 			\right)
 			+\;
 			t\!\sum_{j=2t}^{s-2t+2} [2j+1]
 			+\;
 			\sum_{k=0}^{\frac{s-2t}{2}} [4k+1].
 		\end{equation*}
 	\end{enumerate}
 \end{theorem}
 
 \begin{proof}
The set of positive roots of \( \GL_3(\C)\) is $\Phi^+ = \{\alpha, \beta, \alpha + \beta\}$,
where \(\alpha\) and \(\beta\) are the simple roots. The corresponding coroots are $(\Phi^+)^\vee = \{\alpha^\vee, \beta^\vee, \alpha^\vee + \beta^\vee\}$,
since \((\alpha + \beta)^\vee = \alpha^\vee + \beta^\vee\) due to the simply-laced root system of $\GL_3(\C)$.
Let \(\omega_1, \omega_2\) denote the fundamental weights. Let $\lambda = a\omega_1 + b\omega_2$, for $a,b \ge 0$, be a dominant integral weight.
We have $\rho = \omega_1 + \omega_2$ and hence $\lambda + \rho = (a+1)\omega_1 + (b+1)\omega_2$.
Using the standard pairing \(\langle \omega_i, \alpha_j^\vee \rangle = \delta_{ij}\), we compute
\[
\langle \rho, \alpha^\vee \rangle = 1, 
\qquad
\langle \rho, \beta^\vee \rangle = 1, 
\qquad
\langle \rho, (\alpha+\beta)^\vee \rangle = 2,
\]
\[
\langle \lambda + \rho, \alpha^\vee \rangle = a+1, 
\qquad
\langle \lambda + \rho, \beta^\vee \rangle = b+1
\qquad
\langle \lambda + \rho, (\alpha+\beta)^\vee \rangle = a+b+2.
\]
Hence by \Cref{prop dp char}, the character 
\begin{equation}\label{eqn char formula for gl3}
	\Theta_{\lambda}(z)=	\frac{\left(z^{a+b+2}-z^{-(a+b+2)}\right) \left(z^{b+1}-z^{-(b+1)} \right) \left(z^{a+1}-z^{-(a+1)} \right)}{(z-z^{-1})(z-z^{-1})(z^2-z^{-2})}.
\end{equation}
 
We reorganize the above expression by pairing each denominator factor with a suitable numerator factor, thereby expressing the entire character as a product of three $\SL_2(\C)$-characters(or virtual character).
For the factor $(z^2-z^{-2})$, we pair it with a numerator term of the form $z^{m}-z^{-m}$ where $m$ is even. Among the three integers $a+1,b+1$ and $a+b+2$,
at least one is even for any parity of \(a,b \in \mathbb{Z}_{\ge 0}\). 

Suppose \(b=0\). Then according to the parity of \(a\), character simplifies to
\[
\Theta_{\lambda}(z)
=
\frac{(z^{a+2}-z^{-(a+2)})(z^{a+1}-z^{-(a+1)})}
{(z-z^{-1})(z^2-z^{-2})}=\begin{cases}\frac{z^{a+2}-z^{-(a+2)}}{z^2-z^{-2}}
\cdot \Sym^a & \text{ if } a \text{ is even},\\[.5em]
\frac{z^{a+1}-z^{-(a+1)}}{z^2-z^{-2}}
\cdot \Sym^{a+1}  & \text{ if } a \text{ is odd}.
 \end{cases}.
\]
Therefore, the desired conclusion follows from Part~(1) of \Cref{lem: prod sym with alternate sym}.

Suppose \(b \geq 1\). 
Depending on the parity of \(a\) and \(b\), the character \(\Theta_{\lambda}(z)\) can be expressed in the form
 	\begin{equation*}
 		\Theta_{\lambda}(z)=
 		\frac{z^{m_1}-z^{-m_1}}{z-z^{-1}}
 		\cdot
 		\frac{z^{m_2}-z^{-m_2}}{z-z^{-1}}
 		\cdot
 		\frac{z^{m_3}-z^{-m_3}}{z^2-z^{-2}},
 	\end{equation*}
 	where the parameters \((m_1,m_2,m_3)\) are chosen according to the parity conditions:
 	\[
 	(m_1,m_2,m_3)=
 	\begin{cases}
 		(a+1,a+b+2,b+1), & \text{if } b \text{ is odd},\\
 		(b+1,a+b+2,a+1), & \text{if } b \text{ is even and } a \text{ is odd},\\
 		(a+1,b+1,a+b+2), & \text{if both } a \text{ and } b \text{ are even}.
 	\end{cases}
 	\]
 	Since \(m_3\) is always even, the last factor admits an expansion as an alternating sum of symmetric powers by \eqref{eqn: alter char for div 2 }. Hence
 	\[
 	W:=\frac{z^{m_3}-z^{-m_3}}{z^2-z^{-2}}=\mathrm{Sym}^{m_3-2}
 	-\mathrm{Sym}^{m_3-4}
 	+\cdots
 	+(-1)^{\frac{m_3-4}{2}}\mathrm{Sym}^{2}
 	+(-1)^{\frac{m_3-2}{2}}\mathrm{Sym}^{0}.
 	\]
 Thus, the problem reduces to understanding the tensor product of two symmetric powers together with the alternating sum \(W\). More precisely, we rewrite \(\Theta_{\lambda}(z)\) as
 	\begin{equation}\label{eq three product}
 		\Theta_{\lambda}(z)
 		=
 		\mathrm{Sym}^{m_1-1}
 		\otimes
 		\mathrm{Sym}^{m_2-1}
 		\otimes
 		W.
 	\end{equation}
 	
  The rest of the proof is essentially identical in all cases except when both \(a\) and \(b\) are even. We therefore treat this exceptional case separately, and first establish the result for the remaining two cases. 
  In the first two cases, we first consider the tensor product $\Sym^{m_2-1} \otimes W$
  and then tensor each resulting term with \(\Sym^{m_1-1}\).
  However, in the even-even case, we proceed differently: we first compute $\Sym^{m_1-1} \otimes \Sym^{m_2-1}$ and then tensor each resulting term with \(W\).
  
  For the case when \(b\) is odd, applying part~(1) of Lemma~\ref{lem: prod sym with alternate sym} yields the decomposition of $\mathrm{Sym}^{m_2-1} \otimes W$.
  Substituting this into \eqref{eq three product}, we obtain the desired character:
  \begin{equation}\label{eqn b is odd}
  	\Theta_{\lambda}(z)
  	= \mathrm{Sym}^{a} \otimes \left[\mathrm{Sym}^{a+2} 
  	+ \mathrm{Sym}^{a+6} 
  	+ \cdots 
  	+ \mathrm{Sym}^{a+2b}\right].
  \end{equation}
  The desired conclusion in this case now follows from part~(1) of \Cref{lem: sym{l} with sym{l+2} onwards}. For the case when \(a\) is odd and \(b\) is even, the argument proceeds similarly. In this case, substituting the decomposition of $\mathrm{Sym}^{m_2-1} \otimes W$ into \eqref{eq three product} gives \begin{equation}\label{eqn a odd b even} \Theta_{\lambda}(z) = \mathrm{Sym}^{b} \otimes \left[\mathrm{Sym}^{b+2} + \mathrm{Sym}^{b+6} + \cdots + \mathrm{Sym}^{b+2a}\right]. \end{equation} The result therefore follows from part~(2) of \Cref{lem: sym{l} with sym{l+2} onwards}.
 
 For the even-even case, applying the Clebsch--Gordan formula gives the decomposition of $\Sym^{m_1-1} \otimes \Sym^{m_2-1}$.
 Substituting this into \eqref{eq three product}, we obtain
 \begin{align*}
 	\Theta_{\lambda}(z)
 	&=\left[ \Sym^{a+b} + \Sym^{a+b-2} + \cdots  + \Sym^{a-b} \right]\otimes W \\
 	&=\left[\left( \Sym^{a+b} + \Sym^{a+b-2}\right) + \cdots + \left( \Sym^{a-b-4} + \Sym^{a-b-2}\right) + \Sym^{a-b} \right]\otimes W .
 \end{align*}
 We group the terms as above and tensor each pair with \(W\). The remaining unpaired term is \(\Sym^{s-4t}\), which is then tensored separately with \(W\).
   Summing the contributions from the paired terms after tensoring with \(W\), we obtain 
  \begin{align}
  &	\left[\left( \Sym^s + \Sym^{s-2}\right) + \cdots + \left( \Sym^{s-4t+4} + \Sym^{s-4t+2}\right)\right]\otimes W \notag \\
  =&\left[  \sum_{k=1}^{t-1} k\!\left(
  [4k-1] + [4k+1] 
  \right) 
  +\;
  t\!\sum_{j=2t-1}^{s-2t+2} [2j+1]\right] +\; \sum_{k=1}^{t-1} k\!\left(
  	[r-4k+1] + [r-4k+3]
  	\right) \label{eq even even 2nd step}
  \end{align}
   We use the notation $[k+1]$ for $\Sym^k$, in the last step. The remaining unpaired term contributes
  \begin{align} \label{eq even even another eqn}
  	\Sym^{s-4t} \otimes W=\sum_{j\geq 0} \Sym^{2s-4t-4j} + \left[ - \sum_{j\geq 0} \Sym^{4t-2-4j} \right] 
  \end{align}
Adding \eqref{eq even even 2nd step} and \eqref{eq even even another eqn}, we obtain the desired result. The unbracketed terms in the above two expressions remain unchanged, while summing the bracketed terms (after rearranging) 
\begin{align*}
=&	\left(  \sum_{k=1}^{t-1} k\!\left(
	[4k-1] + [4k+1] 
	\right) + t [4t-1]  -  \sum_{j\geq 0} [4t-1-4j]
	 \right) +\;
	 t\!\sum_{j=2t}^{s-2t+2} [2j+1]\\
	 =&\sum_{k=1}^{t-1} k\!\left(
	 [4k+1] + [4k+3] 
	 \right) +\;
	 t\!\sum_{j=2t}^{s-2t+2} [2j+1]
\end{align*} 
This completes the proof.
\end{proof} 
 
If  $\pi_{\lambda}$ restricted to principal $\SL_2(\C)$ is $\sum k_i [i]$ then we call $\sum k_i$ be the sum of  multiplicities.
 \begin{corollary}
 	The sum of the multiplicities of the Jordan sizes given in Theorem~\ref{thm for gl3} is 
 	$$M=\begin{cases}
 		\frac{(a+1)(b+1) +1}{2} & \text{ if }  a,b \text{ both even} \\
 		\frac{(a+1)(b+1)}{2} & \text{ otherwise } . 
 	\end{cases}$$
 	In particular, an irreducible representation of $\GL_3(\C)$ restricted to the principal $\SL_2(\C)$ is irreducible if and only if $(a,b)=(1,0)$ or $(0,1)$, thus if and only if the representation $\pi_{\lambda}$ is either the standard $3$ dimensional representation of $\GL_3(\C)$ or it's dual.
 \end{corollary}
 \begin{proof}
 	Note that the total multiplicity is the same as counting how many terms are there in \eqref{eqn b is odd} or \eqref{eqn a odd b even}. Tensor product $\Sym^{u}\otimes \Sym^{u+2k}$ contains $u+1$ many components. In those two equations, $k$ is an odd integer that lies between 1 and $v$, where $v$ is an odd integer. So total components is $(u+1)\frac{v+1}{2}$. For even case one counts similarly.
 \end{proof}

 \begin{remark}\label{rem: palindromic}
 	The multiplicity sequence of Jordan sizes appearing in the decomposition depends on the parity of $a$ and $b$ of the highest weight $\lambda=a \omega_1 +b \omega_2$.
 	When $b$ is odd (Figure~\ref{fig:mult_seq_odd_even}:(B)), the sequence exhibits a particularly regular pattern: it starts at $1$, each value appears exactly twice, increases stepwise to a maximum $m$, remains constant for a plateau, and then decreases symmetrically. In particular, the multiplicity sequence is palindromic and has the schematic form
 	\[
 	\underbrace{1,1},\;\underbrace{2,2},\;\ldots,\;\underbrace{m-1,m-1},\;
 	\underbrace{m,\ldots,m}_{\text{plateau}},\;
 	\underbrace{m-1,m-1},\;\ldots,\;\underbrace{2,2},\;\underbrace{1,1}.
 	\]
 	
 	When $a$ is odd and $b$ is even, the sequence remains palindromic and largely follows the same pattern. The only distinction arises at the plateau, it is replaced by an alternating two-level pattern (Figure~\ref{fig:mult_seq_odd_even}:(A)) at the maximum multiplicity. Note that when both $a$ and $b$ are even, the result is not well-structured.
 \end{remark}

 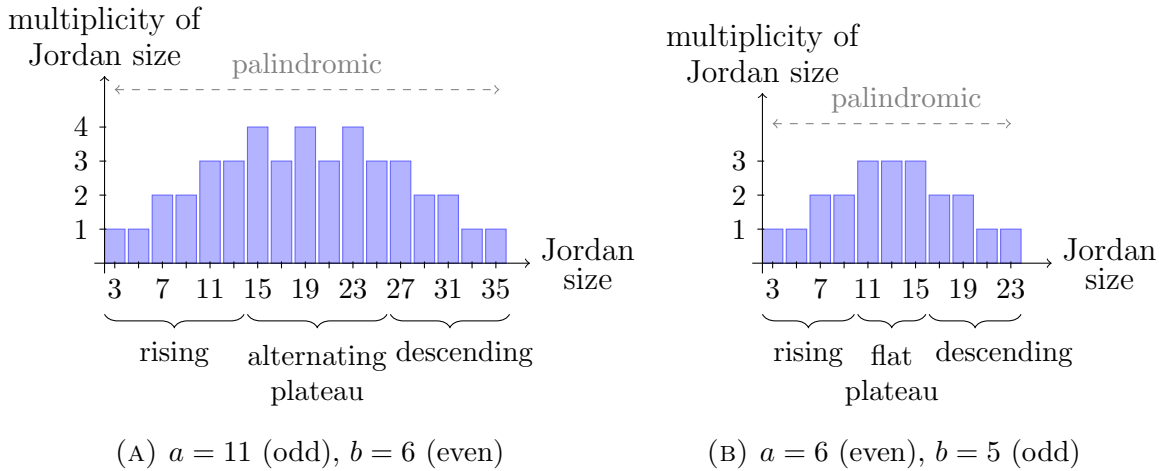
\begin{figure}[h]
 	\centering
 	\begin{subfigure}[b]{0.55\textwidth}
 		\centering
 		\begin{tikzpicture}[scale=0.45]
 			\def\seq{{1,1,2,2,3,3,4,3,4,3,4,3,3,2,2,1,1}}
 			\foreach \i in {0,...,16} {
 				\pgfmathsetmacro{\val}{\seq[\i]}
 				\pgfmathsetmacro{\xpos}{\i * 0.7}
 				\draw[fill=blue!30, draw=blue!60] (\xpos, 0) rectangle (\xpos+0.6, \val);
 			}
 			\draw[->] (-0.3, 0) -- (12.5, 0) node[right] {\shortstack{Jordan\\size}};
 			\draw[->] (0, -0.3) -- (0, 5.5) node[above] {\shortstack{multiplicity of \\ Jordan size}};
 			\foreach \y in {1,2,3,4} {
 				\draw (0.05,\y) -- (-0.15,\y) node[left] {\small\y};
 			}
			\foreach \x in {0.3,1.0,1.7,2.4,3.1,3.8,4.5,5.2,5.9,6.6,7.3,8.0,8.7,9.4,10.1,10.8,11.5}
			{
				\draw (\x,0.05) -- (\x,-0.15);
			}
			
			\foreach \x/\lab in {
				0.3/3,
				1.7/7,
				3.1/11,
				4.5/15,
				5.9/19,
				7.3/23,
				8.7/27,
				10.1/31,
				11.5/35
			}
			{
				\node[below] at (\x,-0.15) {\small\lab};
			}
 			\draw[decorate, decoration={brace, amplitude=5pt, mirror}, yshift=-1.5cm]
 			(0,0) -- (4.1,0) node[midway, below=6pt] {\small rising};
 			\draw[decorate, decoration={brace, amplitude=5pt, mirror}, yshift=-1.5cm]
 			(4.2,0) -- (8.3,0) node[midway, below=6pt] 
 			{\small \begin{tabular}{c}alternating\\plateau\end{tabular}};
 			\draw[decorate, decoration={brace, amplitude=5pt, mirror}, yshift=-1.5cm]
 			(8.4,0) -- (11.9,0) node[midway, below=6pt] {\small \quad descending};
 			\draw[<->, dashed, gray] (0.3, 5.1) -- (11.6, 5.1)
 			node[midway, above] {\small palindromic};
 		\end{tikzpicture}
 		\caption{$a = 11$ (odd), $b = 6$ (even)}
 		\label{fig:mult_seq_alternating}
 	\end{subfigure}
 	\hfill
 	\begin{subfigure}[b]{0.42\textwidth}
 		\centering
 		\begin{tikzpicture}[scale=0.45]
 			\def\seq{{1,1,2,2,3,3,3,2,2,1,1}}
 			\foreach \i in {0,...,10} {
 				\pgfmathsetmacro{\val}{\seq[\i]}
 				\pgfmathsetmacro{\xpos}{\i * 0.7}
 				\draw[fill=blue!30, draw=blue!60] (\xpos, 0) rectangle (\xpos+0.6, \val);
 			}
 			\draw[->] (-0.3, 0) -- (8.5, 0) node[right] {\shortstack{Jordan\\size}};
 			\draw[->] (0, -0.3) -- (0, 5) node[above] {\shortstack{multiplicity of \\ Jordan size}};
 			\foreach \y in {1,2,3} {
 				\draw (0.05,\y) -- (-0.15,\y) node[left] {\small\y};
 			}
 			\foreach \x in {0.3,1.0,1.7,2.4,3.1,3.8,4.5,5.2,5.9,6.6,7.3}
 			{
 				\draw (\x,0.05) -- (\x,-0.15);
 			}
 			
 			\foreach \x/\lab in {
 				0.3/3,
 				1.7/7,
 				3.1/11,
 				4.5/15,
 				5.9/19,
 				7.3/23
 			}
 			{
 				\node[below] at (\x,-0.15) {\small\lab};
 			}
 			\draw[decorate, decoration={brace, amplitude=5pt, mirror}, yshift=-1.5cm]
 			(0,0) -- (2.7,0) node[midway, below=6pt] {\small rising};
 			\draw[decorate, decoration={brace, amplitude=5pt, mirror}, yshift=-1.5cm]
 			(2.8,0) -- (4.8,0) node[midway, below=6pt] {\small \begin{tabular}{c}flat\\plateau\end{tabular}};
 			\draw[decorate, decoration={brace, amplitude=5pt, mirror}, yshift=-1.5cm]
 			(4.9,0) -- (7.6,0) node[midway, below=6pt] {\small \qquad descending};
 			\draw[<->, dashed, gray] (0.3, 4.1) -- (7.3, 4.1)
 			node[midway, above] {\small \quad palindromic};
 		\end{tikzpicture}
 		\caption{$a = 6$ (even), $b = 5$ (odd)}
 		\label{fig:mult_seq_flat}
 	\end{subfigure}
 	\caption{(A) alternating two-level plateau, and (B) flat plateau.}
 	\label{fig:mult_seq_odd_even}
 \end{figure}

\vspace{.5cm}
\noindent{\bf Acknowledgment:} The author thanks Prof. Dipendra Prasad for suggesting this problem and for many fruitful discussions. The author is also grateful to the participants of the Algebraic Groups Seminar on \cite{nilpotent-orbits-book} at IIT Bombay for their valuable discussions, which greatly helped the understanding of this work. The author thanks Prof. Santosha Pattanayak for useful discussions.


\end{document}